\documentclass[11pt,leqno]{article}
\usepackage{amsmath,amsfonts,amscd,amsthm}
\begin{document}

\newcommand{\se}{\setcounter{equation}{0}}
\def\theequation{\thesection.\arabic{equation}}

\newtheorem{theorem}{Theorem}[section]
\newtheorem{cdf}{Corollary}[section]
\newtheorem{lemma}{Lemma}[section]
\newtheorem{remark}{Remark}[section]

\newcommand{\e}{{\bf e}}
\newcommand{\f}{{\bf f}}
\newcommand{\R}{\mathbb{R}}

\newcommand{\bu}{{\bf u}}
\newcommand{\bv}{{\bf v}}
\newcommand{\bw}{{\bf w}}
\newcommand{\by}{{\bf y}}
\newcommand{\bz}{{\bf z}}
\newcommand{\bH}{{\bf H}}
\newcommand{\bJ}{{\bf J}}
\newcommand{\bL}{{\bf L}}

\newcommand{\bchi}{\mbox{\boldmath $\chi$}}
\newcommand{\bta}{\mbox{\boldmath $\eta$}}
\newcommand{\bphi}{\mbox{\boldmath $\phi$}}
\newcommand{\bth}{\mbox{\boldmath $\theta$}}
\newcommand{\bxi}{\mbox{\boldmath $\xi$}}
\newcommand{\bzt}{\mbox{\boldmath $\zeta$}}

\newcommand{\hf}{\hat{\f}}
\newcommand{\hta}{\hat{\bta}}
\newcommand{\hth}{\hat{\bth}}
\newcommand{\hbu}{\hat{\bu}}
\newcommand{\hxi}{\hat{\bxi}}
\newcommand{\hby}{\hat{\by}}
\newcommand{\hzt}{\hat{\bzt}}

\newcommand{\td}{\tilde{\Delta}}
\newcommand{\bub}{\bu_{\beta}}
\newcommand{\bubt}{\bu_{\beta,t}}

\newcommand{\bbu}{\bar{\bu}}
\newcommand{\bby}{\bar{\by}}
\newcommand{\bbz}{\bar{\bz}}

\title
{ \large\bf A study of Nonlinear Galerkin Finite Element for time-dependent \\
incompressible Navier-Stokes equations}

\author{Deepjyoti Goswami\footnote{Department of Mathematical Sciences, School of Sciences,
Tezpur University, Napaam, Sonitpur, Assam, India-784028. E-mail: deepjyoti@tezu.ernet.in}}

\date{}
\maketitle

\begin{abstract}
In this article, we discuss a couple of nonlinear Galerkin methods (NLGM) in finite element set
up for time dependent incompressible Navier-Sotkes equations. We show the crucial role played by
the non-linear term in determining the rate of convergence of the methods. We have obtained 
improved error estimate in $\bL^2$ norm, which is optimal in nature, for linear
finite element approximation, in view of the error estimate available in literature, in
$\bH^1$ norm.
\end{abstract}

\vspace{0.30cm} 

\vspace{.2in}
\section{Introduction}
In the study of dynamical systems, generated by evolution partial differential equations, we look
into long time behavior of the solutions. In certain cases, solutions converge asymptotically to
a compact set called global attractor. Owing to complicated structures of such sets, notion of
inertial manifold (IM) was introduced, which is a smooth finite dimensional manifold, attracting
all the solution trajectories exponentially. Once we can prove the existence of an IM for a
evolution equation, it is far easier to study it than to study the global attractor.
Unfortunately, the existence of an IM for Navier-Stokes equations could not be established,
even for a simpler case like $2D$ spatially periodic flow. Hence came the concept of artificial
inertial manifold (AIM), a sequence of smooth and finite dimensional manifolds of increasing
dimensions and that the global attractor lies in a small neighborhood of each such manifold with
the distance vanishing exponentially as the dimension increases. AIM has been shown to exist for
the Navier-Stokes equations in $2D$. A numerical technique based on AIM was introduced by Marion
and Temam in \cite{MT89} ans was coined non-linear Galerkin method (NLGM). The method involves
splitting the solution of a dissipative system into large and small scales, simplifying the
small scale equation, thereby obtaining small scale in terms of large scale, relatively easily.
In other words, small scales are made slaves to large scales. This technique and its
modifications are studied in depth in early nineties. For details and a history of developments, we refer to \cite{GP08, NR97}.

NLGM was originally developed in the context of spectral Galerkin approximations. The method in
\cite{MT89} was based on the eigenvectors of the underlying linear elliptic operator. Later in
\cite{MT90}, it was expanded to more general bases and more specifically to finite elements. But
for non-spectral Galerkin discretizations, very few results are available. The problem was that
the extension to such cases is not natural, since the splitting of discrete solution space into
large and small scales (or low and high frequency modes) is not obvious any longer. Marion and
Temam \cite{MT90} developed various NLGMs based on finite element and later on, Marion and Xu
\cite{MX95} and, Ammi and Martin \cite{AM94} worked on general semi-linear reaction diffusion
and the Navier-Stokes equations with higher convergence rate in $H^1$-norm for a NLGM type
two-level finite element method. There, the strong orthogonality properties of spectral Galerkin
approximation was substituted by weaker $L^2$-orthogonality of finite element approximation.
In \cite{MX95}, Marion and Xu considered a semi-linear evolution equation and showed that the
$H^1$-error estimate to be of $\mathcal{O}(H^3)$ (here $H$ is a space discretizing parameter).
They expected $L^2$-error estimate to be of $\mathcal{O}(H^4)$ and remarked that {\it this is an
open problem}. We prove this higher order estimate in a couple of cases and believe that it can
not be extended to other possible cases.

The advent of NLGMs created a lot of interest, since it outperformed the standard Galerkin
method, which is restricted to the large scales only. Also, it was thought of as well-suited for
turbulence modeling. But Heywood and Rannacher \cite{HR93} argued that it is not turbulence
modeling that is responsible for the better performance of NLGM over Galerkin method. On the
other hand, it is the ability of NLGM to handle better, the Gibb phenomenon induced by higher
order boundary incompatibilities induced by no-slip boundary condition. They substantiated it by
showing that, in periodic domain, both perform identically. Later, Guermond and Prudhomme, in
\cite{GP08}, revisited NLGM and showed that, in case of arbitrary smoothness, NLGM will always
outperform Galerkin method. And it has nothing to do with turbulence modeling, but due to the
well-known fact of superconvergence of elliptic projection in $H^1$-norm. We have also exploited
this simple fact to improve the existing $L^2$-error estimate.

The incompressible time dependent Navier-Stokes equations are given by
\begin{eqnarray}\label{nse}
\frac {\partial \bu}{\partial t}+\bu\cdot\nabla\bu-\nu\Delta\bu+\nabla p=\f(x,t),~~x\in \Omega,~t>0
\end{eqnarray}
with incompressibility condition
\begin{eqnarray}\label{ic}
 \nabla \cdot \bu=0,~~~x\in \Omega,~t>0,
\end{eqnarray}
and initial and boundary conditions
\begin{eqnarray}\label{ibc}
 \bu(x,0)=\bu_0~~\mbox {in}~\Omega,~~~\bu=0,~~\mbox {on}~\partial\Omega,~t\ge 0.
\end{eqnarray}
Here, $\Omega$ is a bounded domain in $\mathbb{R}^2$ with boundary $\partial \Omega$ and $\nu$ 
is the inverse of the Reynolds number. $\bu$ and $p$ stand for the velocity and the pressure
of a fluid occupying $\Omega$, respectively. Initial velocity $\bu_0$ is a solenoidal vector
field and $\f$ is the forcing term.

Error estimations of NLGM in mixed finite element set up for Navier-Stokes' equations are
carried out in \cite{AM94}. The small scales equations carry both nonlinearity and time dependence and the following estimates were established.
\begin{align*}
\|(\bu^h-\bu_h)(t)\|_{\bH^1(\Omega)} \le c(t)H^2,~~~
\|(p^h-p_h)(t)\|_{\bL^2(\Omega)} \le c(t)H^2,
\end{align*}
where $(\bu_h,p_h)$ is the Galerkin approximation and $(\bu^h,p^h)$ is the nonlinear Galerkin
approximation. These results were improved to $\mathcal{O}(H^3)$ and similar result was obtained
in $\bL^2$-norm for velocity approximation, but for semilinear parabolic problems, by Marion and
Xu, in \cite{MX95}. We feel that it is not straight forward to carry forward these results to
Navier-Stokes equations. The proof of \cite{MX95} depends on one important estimate involving
the nonlinear term $f(\bu)$, namely
$$|(f(\by_h+\bz_h)-f(\by_h),\bchi)| \le \|\bz_h\|_{\bL^2(\Omega)}\|\bchi\|_{\bL^2(\Omega)}, $$
where $\by_h$ and $\bz_h$ are the large and small scales of $\bu^h$, and $\bchi$ is any element
from the appropriate space containing small scales, see \cite[pp. 1176]{MX95}.
But it may not be possible to establish similar estimate for the nonlinear term $(\bu\cdot\nabla)\bu$. On the other hand, it is observed that the improvement in the order of convergence is due to the fact that the splitting in space is done on the basis of $\bL^2$
projections, unlike previous approaches, where the splitting is based on hierarchical basis.
Similar approach is adopted in this article.

No further improvement is observed for (piecewise) linear finite element discretization and with
forcing term $\f\in\bL^2(\Omega)$. Results, similar to the ones mentioned above, are observed in
\cite{HL96,HL98}, but for $\f\in\bH^1(\Omega)$ and for Navier-Stokes equations. In \cite{NR97},
various nonlinear Galerkin finite elements are studied in depth for one-dimensional problems and
similar results are obtained, although termed as optimal in nature. For example, the method is
studied for piecewise polynomials of degree $2n-1$. And the error estimates obtained in energy-
norm and $L^2$-norm are as follows
\begin{align*}
\|(\bu_h-\bu^h)(t)\|_{\bH^n(\Omega)} &\le c(t)H^{3n-1}, \\
\|(\bu_h-\bu^h)(t)\|_{\bL^2(\Omega)} &\le c(t)\min\{t^{-1/2}H^{4n-1},H^{7n/2-1}\},
\end{align*}
respectively. For $n=1$, we have the piecewise linear finite element approximation and
error estimates are of $O(H^2)$ and $O(H^{5/2})$, in energy-norm and $L^2$-norm,
respectively. Later on He {\it et. al} have studied NLGM and modified NLGM in both spectral and
finite element set ups, \cite{HL99,HLL99,HWCL03,HMMC04} to name a few. In these articles, fully
discrete NLGMs were considered and were shown to exhibit better convergence rate than Galerkin
method. But none were optimal.

Recently in \cite{LH10}, a new projection is employed for a two-level finite element for
Navier-Stokes equations. Error estimates of $O((\log h)^{1/2}H^3)$ and $O((\log h)^{1/2} H^4)$,
in energy-norm and $L^2$-norm, respectively, are obtained. Logarithmic term appears due to the
use of the finite dimensional case of the Brezis-Gallouet inequality. Note that here the forcing
term is taken in $L^{\infty}(0,T;\bL^2(\Omega))\cap L^2(0,T;\bH_0^1(\Omega))$.

In this article, we have established optimal $\bL^2$-error estimate for a couple of NLGMs, while the forcing term is in $\bL^2$. The
approach here is an heuristic one and we basically highlight the importance of nonlinearity of
the small scales equations as the key to the higher or lower order error estimate. Since these
methods can be considered as two-level methods, this suggests ways of constructions or
limitations of two or multi-level methods for lower-order approximations.

The article is organized as follows. In section $2$, we briefly recall the notion of a suitable
weak solution of the Navier-Stokes equations. Section $3$ deals with the Galerkin approximation
and state the error estimates. In Section $4$, we introduce the NLGMs. Error estimates for the
methods are discussed in Section $5$. Finally, in Chapter $6$, we summarize the results obtained
in this article. 

\section{Preliminaries}
\se

For our subsequent use, we denote by bold face letters the $\mathbb{R}^2$-valued
function space such as
\begin{align*}
 \bH_0^1 = [H_0^1(\Omega)]^2, ~~~ \bL^2 = [L^2(\Omega)]^2.
\end{align*}
Note that $\bH^1_0$ is equipped with a norm
$$ \|\nabla\bv\|= \big(\sum_{i,j=1}^{2}(\partial_j v_i, \partial_j
 v_i)\big)^{1/2}=\big(\sum_{i=1}^{2}(\nabla v_i, \nabla v_i)\big)^{1/2}. $$
Further, we introduce divergence free function spaces:
\begin{align*}
\bJ_1 &= \{\bphi\in\bH_0^1 : \nabla \cdot \bphi = 0\} \\
\bJ= \{\bphi \in\bL^2 :\nabla \cdot \bphi &= 0~~\mbox{in}~~
 \Omega,\bphi\cdot{\bf n}|_{\partial\Omega}=0~~\mbox {holds weakly}\}, 
\end{align*}
where ${\bf n}$ is the outward normal to the boundary $\partial \Omega$ and $\bphi
\cdot {\bf n} |_{\partial\Omega} = 0$ should be understood in the sense of trace
in $\bH^{-1/2}(\partial\Omega)$, see \cite{temam}.
For any Banach space $X$, let $L^p(0, T; X)$ denote the space of measurable $X$
-valued functions $\bphi$ on  $ (0,T) $ such that
$$ \int_0^T \|\bphi (t)\|^p_X~dt <\infty~~~\mbox {if}~~1 \le p < \infty, $$
and for $p=\infty$
$$ {\displaystyle{ess \sup_{0<t<T}}} \|\bphi (t)\|_X <\infty~~~\mbox {if}~~p=\infty. $$
Through out this paper, we make the following assumptions:\\
(${\bf A1}$). For ${\bf g} \in \bL^2$, let the unique pair of solutions $\{\bv
\in\bJ_1, q \in L^2 /R\} $ for the steady state Stokes problem
\begin{align*}
 -\Delta\bv + \nabla q = {\bf g}, ~~
 \nabla \cdot\bv = 0~~\mbox {in}~\Omega,~~~\bv|_{\partial\Omega}=0,
\end{align*}
satisfy the following regularity result
$$  \| \bv \|_2 + \|q\|_{H^1/R} \le C\|{\bf g}\|. $$
Note that for a $C^2$-domain or for a two-dimensional convex polygon, assumption (${\bf A1}$)
holds.
\noindent
(${\bf A2}$). The initial velocity $\bu_0$ and the external force $\f$ satisfy for
positive constant $M_0,$ and for $T$ with $0<T \leq \infty$
$$ \bu_0\in\bJ_1,~\f,\f_t \in L^{\infty} (0, T ;\bL^2)~~~\mbox{with} ~~~
\|\bu_0\|_1 \le M_0,~~{\displaystyle{\sup_{0<t<T} }}\big\{\|\f\|,
 \|\f_t\|\big\} \le M_0. $$

\noindent We now present a weak solution of (\ref{nse})-(\ref{ibc}). Find a pair of functions $\{\bu(t), p(t)\},~t>0,$ such that
\begin{equation}\label{wfh}
\left\{\begin{array}{rcl}
(\bu_t, \bphi) +\nu (\nabla \bu, \nabla \bphi) +(\bu\cdot\nabla \bu, \bphi)
&=& (p, \nabla \cdot \bphi) + (\f,\bphi)~~~\forall \bphi \in {\bf H}_0^1, \\
(\nabla \cdot \bu, \chi) &=& 0~~~\forall \chi \in L^2.
\end{array}\right.
\end{equation}
Equivalently, find  $\bu(t) \in {\bf J}_1,~t>0 $ such that
\begin{equation}\label{wfj}
 (\bu_t, \bphi) +\nu (\nabla \bu, \nabla \bphi )+( \bu \cdot \nabla \bu, \bphi)
 =(\f,\bphi),~\forall\bphi \in {\bf J}_1.
\end{equation}
The existence, uniqueness and regularity of the solution of the 2D time dependent Navier-Stokes
equations can be found in \cite{temam}.

\section{Galerkin Method}
\se

From now on, we denote $h$ with $0<h<1$ to be a real positive discretization
parameter tending to zero. Let  $\bH_h$ and $L_h$, $0<h<1$ be two family of
finite dimensional subspaces of $\bH_0^1 $ and $L^2/\R$, respectively,
approximating velocity vector and the pressure. Assume that the following
approximation properties are satisfied for the spaces $\bH_h$ and $L_h$: \\
${\bf (B1)}$ For each $\bw \in\bH_0^1 \cap \bH^2 $ and $ q \in
H^1/\R$ there exist approximations $i_h w \in \bH_h $ and $ j_h q \in
L_h $ such that
$$ \|\bw-i_h\bw\|+ h \| \nabla (\bw-i_h \bw)\| \le K_0 h^2 \| \bw\|_2,
 ~~~~\| q - j_h q \| \le K_0 h \| q\|_1. $$
Further, suppose that the following inverse hypothesis holds for $\bw_h\in\bH_h$:
\begin{align}\label{inv.hypo}
 \|\nabla \bw_h\| \leq  K_0 h^{-1} \|\bw_h\|.
\end{align}
For defining the Galerkin approximations, set for $\bv, \bw, \bphi \in \bH_0^1$,
$$ a(\bv, \bphi) = (\nabla \bv, \nabla \bphi),~~~b(\bv, \bw,\bphi)= \frac{1}{2} (\bv \cdot \nabla \bw , \bphi)- \frac{1}{2} (\bv \cdot \nabla \bphi, \bw). $$
Note that the operator $b(\cdot, \cdot, \cdot)$ preserves the antisymmetric property of
the original nonlinear term, that is,
$$ b(\bv_h, \bw_h, \bw_h) = 0 \;\;\; \forall \bv_h, \bw_h \in {\bH}_h. $$
The discrete analogue of the weak formulation (\ref{wfh}) now reads as: Find $\bu_h(t)
\in {\bf H}_h$ and $p_h(t) \in L_h$ such that $ \bu_h(0)= \bu_{0h} $ and for $t>0$
\begin{equation}\label{dwfh}
\left\{\begin{array}{rcl}
(\bu_{ht}, \bphi_h) +\nu a (\bu_h,\bphi_h)+b(\bu_h,\bu_h,\bphi_h)-(p_h, \nabla \cdot \bphi_h) &=& (\f, \bphi_h),~~\bphi_h\in\bH_h \\
(\nabla \cdot \bu_h, \chi_h) &=& 0,~~\chi_h \in L_h.
\end{array}\right.
\end{equation}
Here $\bu_{0h} \in\bH_h $ is a suitable approximation of $\bu_0\in\bJ_1$. For continuous
dependence of the discrete pressure $p_h(t)\in L_h$ on the discrete velocity $u_h(t)\in
{\bf J}_h$, we assume the following discrete inf-sup (LBB) condition for the finite dimensional spaces $\bH_h$ and $L_h$:\\
\noindent
${\bf (B2')}$  For every $q_h\in L_h$, there exists a non-trivial function $\bphi_h\in
{\bf H}_h$ and a positive constant $K_0,$ independent of $h$, such that
$$ |(q_h, \nabla\cdot \bphi_h)| \ge K_0 \|\nabla \bphi_h \|\|q_h\|. $$

\noindent In order to consider a discrete space, analogous to $\bJ_1$, we
impose the discrete incompressibility condition on $\bH_h$ and call it as
$\bJ_h$. Thus, we define $\bJ_h,$ as
$$ {\bf J}_h = \{ v_h \in {\bf H}_h : (\chi_h,\nabla\cdot v_h)=0
 ~~~\forall \chi_h \in L_h \}. $$
Note that $\bJ_h$ is not a subspace of $\bJ_1$. With $\bJ_h$ as above, we now introduce
an equivalent Galerkin formulation as: Find $\bu_h(t)\in {\bf J}_h $ such that $\bu_h(0) =
\bu_{0h} $ and for $\bphi_h \in \bJ_h,t>0$
\begin{equation}\label{dwfj}
~~~~ (\bu_{ht},\bphi_h) +\nu a (\bu_h,\bphi_h)= -b( \bu_h, \bu_h, \bphi_h)+(\f,\bphi_h).
\end{equation}
Since $\bJ_h$ is finite dimensional, the problem (\ref{dwfj}) leads to a system of
nonlinear differential equations. For global existence of a unique solution
of (\ref{dwfj}) (or unique solution pair of (\ref{dwfh})), we refer to \cite{HR82}.

\noindent We further assume that the following approximation property holds
true for ${\bf J}_h $. \\
\noindent
${\bf (B2)}$ For every $\bw \in {\bf J}_1 \cap {\bf H}^2, $ there exists an
approximation $r_h \bw \in {\bf J_h}$ such that
$$ \|\bw-r_h\bw\|+h \| \nabla (\bw - r_h \bw) \| \le K_5 h^2 \|\bw\|_2 . $$
The $L^2$ projection $P_h:\bL^2\mapsto \bJ_h$ satisfies the following properties
: for $\bphi\in \bJ_h$,
\begin{equation}\label{ph1}
 \|\bphi- P_h \bphi\|+ h \|\nabla P_h \bphi\| \leq C h\|\nabla \bphi\|,
\end{equation}
and for $\bphi \in \bJ_1 \cap \bH^2,$
\begin{equation}\label{ph2}
 \|\bphi-P_h\bphi\|+h\|\nabla(\bphi-P_h \bphi)\|\le C h^2\|\td\bphi\|.
\end{equation}
With the definition of the discrete  operator $\Delta_h: \bH_h \mapsto \bH_h$ through the
bilinear form $a (\cdot, \cdot)$,
\begin{eqnarray}\label{do}
 a(\bv_h, \bphi_h) = (-\Delta_h\bv_h, \bphi)~~~~\forall \bv_h, \bphi_h\in\bH_h,
\end{eqnarray}
we set the discrete analogue of the Stokes operator $\td =P(-\Delta) $ as
$\td_h = P_h(-\Delta_h) $. Using Sobolev imbedding and Sobolev inequality, it is
easy to prove the following Lemma
\begin{lemma}\label{nonlin}
Suppose conditions (${\bf A1}$), (${\bf B1}$) and (${\bf  B2}$) are satisfied. Then there 
exists a positive constant $K$ such that for $\bv,\bw,\bphi\in\bH_h$, the following holds:
\begin{equation}\label{nonlin1}
 |(\bv\cdot\nabla\bw,\bphi)| \le K \left\{
\begin{array}{l}
 \|\bv\|^{1/2}\|\nabla\bv\|^{1/2}\|\nabla\bw\|^{1/2}\|\Delta_h\bw\|^{1/2}
 \|\bphi\|, \\
 \|\bv\|^{1/2}\|\Delta_h\bv\|^{1/2}\|\nabla\bw\|\|\bphi\|, \\
 \|\bv\|^{1/2}\|\nabla\bv\|^{1/2}\|\nabla\bw\|\|\bphi\|^{1/2}
 \|\nabla\bphi\|^{1/2}, \\
 \|\bv\|\|\nabla\bw\|\|\bphi\|^{1/2}\|\Delta_h\bphi\|^{1/2}, \\
 \|\bv\|\|\nabla\bw\|^{1/2}\|\Delta_h\bw\|^{1/2}\|\bphi\|^{1/2}
 \|\nabla\bphi\|^{1/2}
\end{array}\right.
\end{equation}
\end{lemma}

The following lemma and theorem present {\it a priori} estimates of the semi-discrete solution
and optimal error estimate, respectively. 
\begin{lemma}\label{est.uh}
Let $0<\alpha <\nu\lambda_1$, where $\lambda_1>0$ is the smallest eigenvalue
of the Stokes' operator. Let the assumptions (${\bf A1}$),(${\bf A2}$),(${\bf B1}$) and (${\bf  B2}$) hold. Then the semi-discrete Galerkin approximation $\bu_h$ of the velocity $\bu$ satisfies, for $t>0,$
\begin{eqnarray}
 \|\bu_h(t)\|+e^{-2\alpha t}\int_0^t e^{2\alpha s}\|\nabla\bu_h(t)\|^2~ds \le K,
 \label{est.uh1} \\
 \|\nabla\bu_h(t)\|+e^{-2\alpha t}\int_0^t e^{2\alpha s}\|\td_h\bu_h(s)\|^2~ds \le K,
 \label{est.uh2} \\
 (\tau^*(t))^{1/2}\|\td_h\bu_h(t)\| \le K, \label{est.uh3}
\end{eqnarray}
where $\tau^*(t)=\min\{t,1\}$. The positive constant $K$ depends only on the given data.
In particular, $K$ is independent of $h$ and $t$.
\end{lemma}

\begin{theorem}\label{errest}
Let $\Omega$ be a convex polygon and let the conditions (${\bf A1}$)-(${\bf A2}$)
and (${\bf B1}$)-(${\bf B2}$) be satisfied. Further, let the discrete initial velocity $\bu_{0h}\in \bJ_h$ with $\bu_{0h}=P_h\bu_0,$ where $\bu_0\in \bJ_1.$ Then,
there exists a positive constant $C$, that depends only on the given data and the
domain $\Omega$, such that for $0<T<\infty $ with $t\in (0,T]$
$$ \|(\bu-\bu_h)(t)\|+h\|\nabla(\bu-\bu_h)(t)\|\le Ce^{Ct}h^2t^{-1/2}. $$
\end{theorem}

A detail account of the finite element spaces, estimates of the semi-discrete solutions and error estimates can be found in \cite{HR82}.

\begin{remark}
From numerical point of view, it is a standard practice to work with mixed formulation
(\ref{dwfh}) rather than (\ref{dwfj}). But following \cite{HR82}, we prefer to base our analysis
on the divergence free formulation (\ref{dwfj}). The same analysis can easily be carried forward
to mixed formulation. It is just a matter of taste.
\end{remark}

\section{Nonlinear Galerkin Method}
\se

In this section, we introduce another space discretizing parameter $H$ such that $0<h<<H$ and
both $h,H$ tend to $0$. And based on that, we split finite element space $\bJ_h$ into two.
\begin{eqnarray}\label{space.split}
\bJ_h=\bJ_H + \bJ_h^H, ~~\mbox{with}~~\bJ_h^H= (I-P_H)\bJ_h
\end{eqnarray}
Note that, by definition, the spaces $\bJ_H$ and $\bJ_h^H$ are orthogonal with respect to the
$\bL^2$-inner product $(\cdot,\cdot)$. In practice, $\bJ_H$ corresponds to a coarse grid and
$\bJ_h^H$ corresponds to a fine grid. In case of mixed method, we only split the velocity space
and pressure space remains the same.

\noindent The following properties are crucial for our error estimates. For a proof,
we refer to \cite{AM94}.
\begin{eqnarray}\label{jhH1}
\|\bchi\| \le cH\|\bchi\|_1,~~\bchi\in\bJ_h^H.
\end{eqnarray}
And there exists $0<\rho<1$ independent of $h$ and $H$ such that
\begin{equation}\label{jhH2}
|a(\bphi,\bchi)| \le (1-\rho)\|\bphi\|_1\|\bchi_h\|_1,~~\bphi\in\bJ_H, \bchi\in\bJ_h^H.
\end{equation}
From (\ref{jhH2}), we can easily deduce that
\begin{equation}\label{jhH3}
\rho\{\|\bphi\|_1^2+\|\bchi\|_1^2\} \le \|\bphi+\bchi_h\|_1^2,~~\bphi\in\bJ_H, \bchi\in\bJ_h^H.
\end{equation}
In nonlinear Galerkin methods, we look for a solution $\bu^h$ ($\in \bJ_h$) in terms of its
components $\by^H$ (coarse grid) and $\bz^h$ (fine grid).
$$ \bu^h=\by^H+\bz^h \in \bJ_H+\bJ_h^H. $$
In the first of the two methods (NLGM I), for $t>t_0>0$, we look for a pair $(\by^H,\bz^h)$ satisfying
\begin{equation}\label{dwfJH}
\left\{\begin{array}{rcl}
(\by_t^H, \bphi) +\nu a (\bu^h,\bphi)+b(\bu^h,\bu^h,\bphi) &=& (\f,\bphi),~~\bphi\in\bJ^H, \\
\nu a (\bu^h,\bchi)+ b(\bu^h,\bu^h,\bchi) &=& (\f,\bchi),~~\bchi\in\bJ_h^H.
\end{array}\right.
\end{equation}
In the second one (NLGM II), we again look for a pair $(\by^H,\bz^h)$ satisfying, for $t>t_0>0$
\begin{equation}\label{dwfJH1}
\left\{\begin{array}{rcl}
(\by_t^H, \bphi) +\nu a (\bu^h,\bphi)+ b(\bu^h,\bu^h,\bphi) &=& (\f,\bphi),~~\bphi\in\bJ^H, \\
\nu a (\bu^h,\bchi) + b(\bu^h,\by^H,\bchi)+b(\by^H,\bz^h,\bchi) &=&(\f,\bchi),~~\bchi\in\bJ_h^H.
\end{array}\right.
\end{equation}
We set $\by^H(t_0)=P_H\bu_h(t_0).$ Note that the two systems differ only by the term $b(\bz^h,
\bz^h, \bchi)$, which has been dropped from the second system, assuming $\bz^h$ is small.
\begin{remark}
Here and henceforth, subscript means the classical Galerkin method and superscript
means the nonlinear Galerkin method.
\end{remark}
\begin{remark}
The nonlinear Galerkin approximations are carried out away from $t=0$. We can employ Galerkin
approximation to obtain $(\bu_h,p_h)$ on the interval $(0,t_0]$. This is done to avoid nonlocal
compatibility conditions. In \cite{HR82}, Heywood and Rannacher showed that to assume higher
regularity for the solution demands some nonlocal compatibility conditions to be satisfied by
initial velocity and initial pressure. These conditions are very difficult to verify and do
not arise in physical context. In order to avoid them, we must admit singularity of velocity
field in a higher order norm at $t=0$, like (1.6) from \cite{HR82}. Since the error analysis of NLGM demands higher regularity of the velocity and this means higher singularity at $t=0$, the idea is to avoid these kinds of singularity by staying away from $t=0$.
\end{remark}
\begin{remark}
Both the NLGs can be heuristically derived from (\ref{dwfj}) as follows.
We split the Galerkin approximation $\bu_h$ with the help of the $\bL^2$ projection $P_H$.
\begin{equation}\label{gss} 
\bu_h=P_H\bu_h+(I-P_H)\bu_h=\by_H+\bz_h, 
\end{equation}
And we project the system (\ref{dwfj}) on $\bJ_H,\bJ_h^H$ to obtain the coupled system:
\begin{equation}\label{dwfjH}
\left\{\begin{array}{rc}
(\by_{ht},\bphi) +& \nu a (\bu_h,\bphi)= -b( \bu_h, \bu_h, \bphi)+(\f,\bphi), \\
(\bz_{ht},\bchi)+& \nu a (\bu_h,\bchi)= -b( \bu_h, \bu_h, \bchi)+(\f,\bchi),
\end{array}\right.
\end{equation}
for $\bphi \in \bJ_H,~\bchi \in \bJ_h^H$.
Assuming the time derivative and higher space derivatives of $\bz_h$ are small, various
(modified) NLG methods are defined. In case, time derivative of $\bz_h$ is retained in the
equation, different time-steps can be employed for the two equations; (much) smaller time-step
for the equation involving $\by_H$, thereby keeping $\bz_h$ remains steady.
\end{remark}

The well-posedness of both the NLGMs (\ref{dwfJH}) and (\ref{dwfJH1}) follows easily from the
works of Marion {\it et. al.} \cite{AM94, MT90, MX95}. For the sake of completeness, we present
below, the {\it a priori} estimates for the approximate solution pair $\{\by^H,\bz^h\}$. And
for the sake of brevity, we only sketch a proof.
\begin{lemma}\label{est1.H}
Under the assumptions of Lemma \ref{est.uh}, the solution pair $(\by^H,\bz^h)$ of
(\ref{dwfjH}) or (\ref{dwfJH1}) satisfy the following estimates, for $t>t_0$
\begin{equation} \label{est1.H1}
\|\by^H\|^2+e^{-2\alpha t}\int_{t_0}^t e^{2\alpha s} \|\nabla(\by^H+\bz^h)\|^2ds \le K.
\end{equation}
And if $H$ is small enough to satisfy
\begin{equation}\label{H.restrict}
\nu-cL_H\|\by^H\|>0,
\end{equation}
where $L_H \sim |log~H|^{1/2}$, then the following estimate holds
\begin{equation}\label{apr2}
\|\bz^h\| \le K.
\end{equation}
The constant $K>0$ depends only on the given data. In particular, $K$ is independent of $h,H$
and $t$.
\end{lemma}

\begin{proof}
Choose $\bphi=\by^H,\bchi=\bz^h$ in (\ref{dwfjH}) or (\ref{dwfJH1}) and add the resulting
equations. Note that the nonlinear terms sum up to 0. Multiply by $e^{2\alpha t}$. Use
Poincar\'e inequality and then kickback argument. Finally, integrate from $t_0$ to $t$ and
multiply by $e^{-2\alpha t}$ to complete the first estimate.
For the second estimate, we put $\bchi=\bz^h$ in the second equation of (\ref{dwfjH}) or
(\ref{dwfJH1}).
\begin{align}\label{est1.H04}
\nu\|\nabla\bz^h\|^2 \le \|\f\|\|\bz^h\|+\nu \|\nabla\by^H\|\|\nabla\bz^h\|+|b(\by^H+\bz^h, \by^H, \bz^h)|.
\end{align}
Using (\ref{inv.hypo}) and (\ref{jhH1}), we find
\begin{align*}
b(\by^H,\by^H,\bz^h) & \le 2^{1/2}\|\by^H\|^{1/2}\|\nabla\by^H\|^{3/2}\|\bz^h\|^{1/2}
\|\nabla\bz^h\|^{1/2} \le c\|\by^H\|\|\nabla\by^H\|\|\nabla\bz^h\| \\
b(\bz^h,\by^H,\bz^h) & \le \|\bz^h\|\|\nabla\bz^h\|\|\by^H\|_{\infty}
\le cL_H\|\bz^h\|\|\nabla\bz^h\|\|\nabla\by^H\|
\end{align*}
We have used the finite dimensional case of Brezis-Gallouet inequality (see \cite[(3.12)]{LH10})
$$ \|\bu_h\|_{\infty} \le cL_h\|\nabla\bu_h\|;~~L_h\sim |log~h|^{1/2}. $$
Incorporate these estimates in (\ref{est1.H04}). With the re-use of (\ref{jhH1}), we have
$$ \nu\|\nabla\bz^h\| \le cH\|\f\|+\nu \|\nabla\by^H\|+c\|\by^H\|\|\nabla\by^H\|+cL_H\|\bz^h\| \|\nabla\by^H\|. $$
Again use (\ref{inv.hypo}), we find
\begin{align*}
\nu\|\bz^h\| \le cH\|\nabla\bz^h\| \le cH^2\|\f\|+c\|\by^H\|+c\|\by^H\|^2+cL_H\|\bz^h\|\|\by^H\|
\end{align*}
Apply (\ref{H.restrict}) and (\ref{est1.H1}), we conclude the proof.
\end{proof}
\begin{remark}
Under the smallness assumption on $H$, that is, (\ref{H.restrict}), we could similarly prove
higher order estimates of $\by^H$ and $\bz^h$. For details, we refer to \cite{AM94, MT90, MX95}.
\end{remark}

\section{Error Estimate}
\se

In this section, we work out the error between classical Galerkin approximation and
nonlinear Galerkin approximation of velocity.

\noindent Before actually working out the error estimates, we present below the Lemma
involving the estimate of $\bz_h$. For a proof, we refer to \cite{AM94}.

\begin{lemma}\label{bzh.est}
Under the assumptions Lemma 3.2 and for the solution $\bu_h$ of (\ref{dwfj}), the following
estimates are satisfied for $\bz_h=(I-P_H)\bu_h$ and for $t>t_0$
\begin{equation}\label{bzh.est1}
\left\{\begin{array}{rcl}
\|\bz_h\|+H\|\bz_h\|_1 \le K(t)H^2, \\
\|\bz_{ht}\|+H\|\bz_{ht}\|_1 \le K(t)H^2, \\
\|\bz_{htt}\|+H\|\bz_{htt}\|_1 \le K(t)H^2.
\end{array}\right.
\end{equation}
\end{lemma}
\noindent
In order to separate the effect of the nonlinearity in the error, we introduce
$$ \bbu (\in\bJ_h)=P_H\bbu+(I-P_H)\bbu =\bby+\bbz\in \bJ_H+\bJ_h^H $$
satisfying the following linearized system ($t>t_0$)
\begin{align}\label{dwfjHl}
(\bby_t,\bphi) +& \nu a(\bbu,\bphi)+ \int_{t_0}^t \beta(t-s) a(\bbu(s), \bphi)~ds
= -b( \bu_h, \bu_h, \bphi)+(\f,\bphi)~~ \bphi \in \bJ_H \nonumber \\
& \nu a (\bbu,\bchi)+ \int_{t_0}^t \beta(t-s) a(\bbu(s), \bchi)~ds
= -b( \bu_h, \bu_h, \bchi)+(\f,\bchi)~~ \bchi \in \bJ_h^H,
\end{align}
and $\bby(t_0)=\by_H(t_0).$ Being linear it is easy to establish the well-posedness
of the above system and the following estimates.
\begin{lemma}
Under the assumptions of Lemma 3.2, we have
\begin{align*}
\|\nabla\bbu^h\|^2+e^{-2\alpha t}\int_{t_0}^t e^{2\alpha t}\|\td_h\bbu^h\|^2ds \le K,~~~
\|\td_h\bbu^h\| \le K,
\end{align*}
where the constant depends on $\bu_0$ and $\f$.
\end{lemma}

\subsection{NLGM I}

\noindent We define
$$ \e :=\bu_h-\bu^h=(\by_H-\by^H)+(\bz_h-\bz^h) =: \e_1+\e_2. $$
We further split the errors as follows:
\begin{align} \label{err.split}
\left\{\begin{array}{rcl}
\e_1 &=& \by_H-\by^H = (\by_H-\bby)-(\by^H-\bby) =\bxi_1-\bta_1 \in\bJ_H \\
\e_2 &=& \bz_h-\bz^h = (\bz_h-\bbz)-(\bz^h-\bbz) =\bxi_2-\bta_2 \in\bJ_h^H.
\end{array}\right.
\end{align}
For the sake of simplicity, we write
$$ \bxi=\bxi_1+\bxi_2,~~~\bta=\bta_1+\bta_2. $$
And the equations in $\bxi$ and $\bta$: subtract (\ref{dwfjHl}) from (\ref{dwfjH})
and subtract (\ref{dwfjHl}) from (\ref{dwfJH}) to obtain
\begin{align} \label{bxi12}
\left\{\begin{array}{rl}
(\bxi_{1,t},\bphi)+& \nu a (\bxi,\bphi)=0 \\
& \nu a (\bxi,\bchi)= -(\bz_{ht},\bchi),
\end{array}\right.
\end{align}
\begin{align} \label{bta12}
\left\{\begin{array}{rl}
(\bta_{1,t},\bphi) +& \nu a (\bta,\bphi)= b(\bu_h,\bu_h,\bphi)-b(\bu^h,\bu^h,\bphi) \\
& \nu a (\bta,\bchi)= b(\bu_h,\bu_h,\bchi)-b(\bu^h,\bu^h,\bchi),
\end{array}\right.
\end{align}
for $\bphi\in\bJ_H$ and $\bchi\in\bJ_h^H.$
\begin{lemma}\label{l2l2.bxi}
Under the assumptions of Lemma \ref{est.uh}, the following holds.
\begin{equation}\label{l2l2.bxi1}
e^{-2\alpha t}\int_{t_0}^t e^{2\alpha \tau}\|\bxi(\tau)\|^2d\tau \le K(t)H^8.
\end{equation}
\end{lemma}

\begin{proof}
Choose $\bphi=e^{2\alpha t}\bxi_1,~\bchi=e^{2\alpha t}\bxi_2$ in (\ref{bxi12}), add the
two resulting equations and with the notation $\hxi=e^{\alpha t}\bxi$, we get
\begin{equation}\label{l2l2.bxi01}
\frac{1}{2}\frac{d}{dt}\|\hxi_1\|^2-\alpha\|\hxi_1\|^2+\nu\|\hxi\|_1^2 \le e^{\alpha t} \|\bz_{ht}\|\|\hxi_2\|.
\end{equation}
Using (\ref{jhH1}) and (\ref{bzh.est1}), we can bound the right-hand side as:
$$ \le e^{\alpha t}.K(t)H^2.cH\|\hxi_2\|_1 \le \frac{\nu\rho}{2}\|\hxi_2\|_1^2
+K(t)H^6.e^{2\alpha t}. $$
And using (\ref{jhH3}), we have
$$ -\alpha\|\hxi_1\|^2+\nu\|\hxi\|_1^2 \ge (\nu\rho-\frac{\alpha}{\lambda_1})
\|\hxi_1\|_1^2+\nu\rho\|\hxi_2\|_1^2. $$
As a result, we obtain from (\ref{l2l2.bxi01})
\begin{equation}\label{l2l2.bxi02}
\frac{d}{dt}\|\hxi_1\|^2+2(\nu\rho-\frac{\alpha}{\lambda_1})\|\hxi_1\|_1^2
+\nu\rho\|\hxi_2\|_1^2 \le K(t)H^6.e^{2\alpha t}.
\end{equation}
Integrate from $t_0$ to $t$ and multiply the resulting inequality by $e^{-2\alpha t}$.
\begin{equation}\label{l2l2.bxi.2}
\|\bxi_1\|^2+e^{-2\alpha t}\int_{t_0}^t (\|\hxi_1\|_1^2+\|\hxi_2\|_1^2)~ds \le K(t)H^6.
\end{equation}
To obtain $L^2(\bL^2)$-norm estimate, we consider the following discrete backward
problem: for fixed $t_0$, let $\bw(\tau)\in \bJ_h,~\bw=\bw_1+\bw_2$ such that $\bw_1
\in \bJ_H,~\bw_2\in \bJ_h^H$ be the unique solution of ($t_0\le \tau <t$)
\begin{equation}\label{bp1} 
\left\{\begin{array}{rcl}
(\bphi,\bw_{1,\tau}) &-&\nu a(\bphi,\bw) = e^{2\alpha \tau}(\bphi,\bxi_1) \\
&-& \nu a(\bchi,\bw) = e^{2\alpha \tau}(\bchi,\bxi_2) \\
&& \bw_1(t)=0.
\end{array}\right.
\end{equation}
With change of variable, we can make it a forward problem and it turns out to be a linearized
version of (\ref{dwfJH}) and hence is well-posed. The following regularity result holds.
\begin{equation}\label{bp1.est}
\int_{t_0}^t e^{-2\alpha \tau}\|\bw\|_2^2d\tau \le C\int_{t_0}^t \|\hxi\|^2d\tau.
\end{equation}
Now, choose $\bphi=\bxi_1,~\bchi=\bxi_2$ in (\ref{bp1}) and use (\ref{bxi12}) with $\bphi=\bw_1,
~\bchi=\bw_2$ to find that
\begin{align*}
\|\hxi(\tau)\|^2 &= (\bxi_1,\bw_{1,\tau})-\nu a (\bxi,\bw) \\
&\le \frac{d}{dt} (\bxi_1,\bw_1)+(\bz_{ht},\bw_2).
\end{align*}
Integrate from $t_0$ to $t$.
\begin{equation}\label{l2l2.bxi03}
\int_{t_0}^t \|\hxi(\tau)\|^2d\tau =((\bxi_1(t),\bw_1(t))-(\bxi_1(t_0),\bw_1(t_0))
+\int_{t_0}^t (\bz_{ht},\bw_2) d\tau.
\end{equation}
But $\bw_1(t)=0$ and $\bxi_1(t_0)=\by_H(t_0)-\bby(t_0)=0$. Next, we observe that
\begin{equation*}
\left.\begin{array}{c}
\bw_1\in\bJ_H \implies P_H\bw_1=\bw_1, ~~\mbox{i.e.,}~~ \bw_1-P_H\bw_1=0 \\
\bw_2\in\bJ_h^H \implies P_H\bw_2=0, ~~\mbox{i.e.,}~~ \bw_2-P_H\bw_2=\bw_2
\end{array}\right\}
\implies \bw-P_H\bw=\bw_2
\end{equation*}
Therefore,
\begin{align}\label{new.est}
(\bz_{ht},\bw_2) &=(\bz_{ht},\bw-P_H\bw) \nonumber \\
& \le \|\bz_{ht}\|\|\bw-\Phi_H\bw\| \le K(t)H^2.cH^2\|\bw\|_2.
\end{align}
From (\ref{l2l2.bxi03}), we get
$$ \int_{t_0}^t \|\hxi(\tau)\|^2d\tau \le K(t) e^{2\alpha t}H^4\Big(\int_{t_0}^t
e^{-2\alpha \tau}\|\bw\|_2^2d\tau\Big)^{1/2}. $$
Use (\ref{bp1.est}) to conclude.
\end{proof}
In order to obtain optimal $L^{\infty}(\bL^2)$ estimate, we would like to introduce
Stokes-type projections $(S_H,S_h^H)$ for $t>t_0$ defined as below:
$$ S_H:\bJ_h\to \bJ_H,~~S_h^H:\bJ_h\to \bJ_h^H, $$
and with the notations
$$\bzt_1 :=\by_H-S_H\bu_h \in \bJ_H,~~\bzt_2 :=\bz_h-S_h^H\bu_h\in\bJ_h^H $$
the following system is satisfied.
\begin{align} \label{bzt12}
\left\{\begin{array}{rl}
\nu a (\bzt,\bphi)=0,~\bphi\in\bJ_H, \\
\nu a (\bzt,\bchi)= -(\bz_{ht},\bchi),~\bchi\in\bJ_h^H.
\end{array}\right.
\end{align}
For the sake of convenience, we have written $\bzt=\bzt_1+\bzt_2.$ Note that, given a
semi-discrete Galerkin approximation $\bu_h$ of NSE with {\it a priori} estimates, the system
(\ref{bzt12}) is a Stokes system, with Stokes problem in $\bJ_h$ projected to subspaces $\bJ_H$
and $\bJ_h^H$, and hence is well-posed.
\begin{lemma}\label{l2.bzt}
Under the assumptions of Lemma \ref{est.uh}, we have
\begin{equation}\label{l2.bzt1}
\|\bzt\|+\|\bzt_t\| \le K(t)H^4.
\end{equation}
\end{lemma}

\begin{proof}
Choose $\bphi= e^{2\alpha t}\bzt_1,~\bchi=e^{2\alpha t}\bzt_2$ in (\ref{bzt12}) to obtain
\begin{equation}\label{l2.bzt01}
 \nu\|\hzt\|_1^2 \le e^{\alpha t}\|\bz_{ht}\|\|\hzt_2\|.
\end{equation}
As in (\ref{l2l2.bxi01}), we establish
\begin{equation}\label{l2.bzt3}
\|\bzt\|_1^2 \le \|\bzt_1\|_1^2+\|\bzt_2\|_1^2 \le K(t)H^6.
\end{equation}
In order to obtain optimal $L^{\infty}(\bL^2)$-norm estimate, we would use
Aubin-Nitsche duality argument. For that purpose, we consider the following
Galerkin approximation of steady Stoke problem: let $\bw_h\in\bJ_h$ be the solution of
$$ \nu a(\bv,\bw_h)=(\bv,\hzt_1+\hzt_2),~\bv\in\bJ_h. $$
Writing $\bw_{h1}=P_H\bw_h,~\bw_{h2}=(I-P_H)\bw_h$, we split the above equation as
\begin{equation}\label{dual.prob1}
\left\{\begin{array}{rcl}
\nu a(\bphi,\bw_{h1}) &=& (\bphi,\hzt_1),~\bphi\in\bJ_H, \\
\nu a(\bchi,\bw_{h2}) &=& (\bchi,\hzt_2),~\bchi\in \bJ_h^H.
\end{array}\right.
\end{equation}
Standard elliptic regularity leads us to the following result.
\begin{equation}\label{dual.reg1}
\|\bw_h=\bw_{h1}+\bw_{h2}\|_2 \le c\|\hzt_1+\hzt_2\|. 
\end{equation}
Now, put $\bphi=\hzt_1,~\bchi=\hzt_2$ in (\ref{dual.prob1}) and use (\ref{bzt12})
with $\bphi=\bw_{h1},~\bchi=\bw_{h2}$ to find that
\begin{align*}
\|\hzt\|^2 =\nu a(\hzt,\bw_h)= -e^{2\alpha t}(\bz_{ht},\bw_{h2}).
\end{align*}
As in (\ref{new.est}) along with (\ref{dual.reg1}), we find
$$ \|\bzt\| \le K(t)H^4. $$
For the remaining part, we differentiate (\ref{bzt12}) and proceed as above to complete the rest of the proof.
\end{proof}
Now we are in a position to estimate $L^{\infty}(\bL^2)$-norm of $\bxi$, that is,
of $\bxi_1$ and $\bxi_2$. Using the definitions of $\bxi_i,\bzt_i,~i=1,2$, we write
\begin{equation*}
\left\{\begin{array}{rcl}
\bxi_1 &=& \by_H-\bby= (\by_H-S_H\bu_h)-(\bby-S_H\bu_h) =:\bzt_1-\bth_1, \\
\bxi_2 &=& \bz_h-\bbz= (\bz_h-S_h^H\bu_h)-(\bbz-S_h^H\bu_h) =:\bzt_2-\bth_2.
\end{array}\right.
\end{equation*}
From (\ref{bxi12}) and (\ref{bzt12}), we have
\begin{equation}\label{bth12}
\left\{\begin{array}{rcl}
(\bth_{1,t},\bphi) &+& \nu a(\bth,\bphi)=(\bzt_{1,t},\bphi),~\bphi\in\bJ_H, \\
&&\nu a(\bth,\bchi)=0,~\bchi\in\bJ_h^H.
\end{array}\right.
\end{equation}
\begin{lemma}\label{l2.bxi}
Under the assumptions Lemma \ref{est.uh}, we have
$$ \|\bxi\| \le K(t)H^4. $$
\end{lemma}

\begin{proof}
Put $\bphi=e^{2\alpha t}\bth_1,~\bchi=e^{2\alpha t}\bth_2$ in (\ref{bth12}) to find
\begin{equation}\label{l2.bxi01}
\frac{1}{2}\frac{d}{dt}\|\hth_1\|^2-\alpha\|\hth_1\|^2+\nu\|\hth\|_1^2  \le e^{\alpha t} \|\bzt_{1,t}\|\|\hth_1\|.
\end{equation}
We recall that the spaces $J_H$ and $\bJ_h^H$ are orthogonal in $\bL^2$-inner product.
That is,
$$ \mbox{for}~\bphi\in\bJ_H,~\bchi\in\bJ_h^H,~~(\bphi,\bchi)=0. $$
Hence
$$ \|\hth_1\|^2 \le \|\hth_1\|^2+\|\hth_2\|^2= \|\hth\|^2 \le \|\hxi\|^2+\|\hzt\|^2,
~~\|\bzt_{1,t}\|^2 \le \|\bzt_t\|^2. $$
And
$$ -\alpha\|\hth_1\|^2+\nu\|\hth\|_1^2 =(\nu-\alpha\lambda_1)\|\hth_1\|_1^2
+\nu\|\hth_2\|_1^2. $$
As a result, after integrating (\ref{l2.bxi01}) with respect to time from $t_0$ to
$t$, we obtain
\begin{equation}\label{l2.bxi02}
\|\hth_1\|^2+\int_{t_0}^t \big(\|\hth_1\|_1^2+\|\hth_2\|_1^2\big)~ds \le \Big(\int_{t_0}^t
e^{2\alpha t}\|\bzt_t\|^2ds\Big)^{1/2}\Big(\int_{t_0}^t \big(\|\hxi\|^2
+\|\hzt\|^2\big)~ds\Big)^{1/2}.
\end{equation}
We now use Lemmas \ref{l2l2.bxi} and \ref{l2.bzt} to conclude from (\ref{l2.bxi02}) that
\begin{equation}\label{l2.bxi2}
\|\bth_1\|^2+e^{-2\alpha t}\int_{t_0}^t e^{2\alpha s}\big(\|\bth_1\|_1^2 +\|\bth_2\|_1^2\big)~ds \le K(t)H^8.
\end{equation}
We again choose $\bchi=e^{2\alpha t}\bth_2$ in (\ref{bth12}) to find
\begin{align*}
\nu \|\hth_2\|_1^2= -\nu a(\hth_1,\hth_2)\implies \|\bth_2\|_1 \le \|\bth_1\|_1.
\end{align*}
Since $\bth_1\in\bJ_H$, we use inverse inequality (\ref{inv.hypo}) and (\ref{l2.bxi2}) to note
that
$$ \|\bth_1\|_1 \le cH^{-1}\|\bth\| \le K(t)H^3. $$
Hence, we conclude that
$$ \|\bth_2\|_1 \le K(t)H^3. $$
Now use (\ref{jhH1}) to see that
\begin{equation}\label{l2.bxi3}
\|\bth_2\| \le K(t)H^4.
\end{equation}
Combining (\ref{l2.bxi2})-(\ref{l2.bxi3}), we establish
$$ \|\bth\| \le K(t)H^4. $$
Use triangle inequality and the estimates of $\bzt$ and $\bth$ to complete the proof.
\end{proof}
\noindent
We are now left with the estimate of $\bta$, the error due to the nonlinearity.
\begin{lemma}\label{l2.err}
Under the assumptions of Lemma \ref{est.uh} and that $H$ is small enough to satisfy (\ref{H.restrict}) and
$$ \mu\rho-2H\|\bbu\|_2 \ge 0,~~\mu-H(\|\bbu\|_2+\|\by^H\|_2) \ge 0, $$
we have
$$ \|(\bu_h-\bu^h)(t)\| \le K(t)H^4. $$
\end{lemma}

\begin{proof}
We choose $\bphi=e^{2\alpha t}\bta_1,~\bchi=e^{2\alpha t}\bta_2$ in (\ref{bta12}).
\begin{align}\label{l2.err01}
\frac{1}{2}\frac{d}{dt}\|\hta_1\|^2+\nu \|\hta\|_1^2= e^{2\alpha t}\Lambda_h(\bta_1,\bta_2),
\end{align}
where
$$ \Lambda_h(\bta_1,\bta_2)=\Lambda_{h,1}(\bta_1)+\Lambda_{h,2}(\bta_2), $$
and
\begin{align*}
\Lambda_{h,1}(\bta_1) &=b(\bu_h,\bu_h,\bta_1)-b(\bu^h,\bu^h,\bta_1) \\
&= b(\bxi-\bta,\bu_h,\bta_1)+b(\bu_h,\bxi-\bta,\bta_1)-b(\bxi-\bta,\bxi-\bta,\bta_1) \\
\Lambda_{h,2}(\bta_2) &=b(\bu_h,\bu_h,\bta_2)-b(\bu^h,\bu^h,\bta_2) \\
&= b(\bxi-\bta,\bu_h,\bta_2)+b(\bu_h,\bxi-\bta,\bta_2)-b(\bxi-\bta,\bxi-\bta,\bta_2).
\end{align*}
Therefore
\begin{align}\label{l2.err02}
\Lambda_h(\bta_1,\bta_2)= b(\bxi-\bta,\bbu,\bta)+b(\bu_h,\bxi,\bta).
\end{align}
\noindent We estimate the nonlinear terms as follows:
\begin{align*}
b(\bu_h,\bxi,\bta)&+b(\bxi-\bta_1,\bbu,\bta) \le \big\{\|\bxi\|\|\bu_h\|_2+
(\|\bxi\|+\|\bta_1\|)\|\bbu\|_2\big\}\|\bta\|_1 , \\
b(\bta_2,\bbu,\bta) &\le \|\bta_2\|_1\|\bbu\|_2(\|\bta_1\|+\|\bta_2\|) \le
\|\bta_2\|_1\|\bbu\|_2\|\bta_1\|+H\|\bbu\|_2\|\bta_2\|_1^2.
\end{align*}
Therefore, for $\epsilon,\epsilon_1>0$,
\begin{align*}
\Lambda_h(\bta_1,\bta_2) &\le \epsilon\|\bta\|_1^2+\epsilon_1\|\bta_2\|_1^2+c(\epsilon) (\|\bu_h\|_2^2+\|\bbu\|_2^2)\|\bxi\|^2 \\
& ~~+c(\epsilon,\epsilon_1)\|\bbu\|_2^2\|\bta_1\|^2+H\|\bbu\|_2\|\bta_2\|_1^2.
\end{align*}
Now, from (\ref{l2.err01}), we find that
\begin{align}\label{l2.err03}
\frac{d}{dt}\|\hta_1\|^2+2\nu\rho &(\|\hta_1\|_1^2+\|\hta_2\|_1^2) \le 2\epsilon \|\hta\|_1^2+2\epsilon_1\|\hta_2\|_1^2 \nonumber \\
&+c(\epsilon) (\|\bu_h\|_2^2+\|\bbu\|_2^2)\|\hxi\|^2+c(\epsilon,\epsilon_1) \|\bbu\|_2^2\|\hta_1\|^2+2H\|\bbu\|_2\|\hta_2\|_1^2.
\end{align}
We choose $\epsilon=\epsilon_1=\nu\rho$ and assume that $H$ small enough such that
$$ \nu\rho-2H\|\bbu\|_2 \ge 0 $$
to obtain after integration
$$ \|\bta_1\|^2+e^{-2\alpha t}\int_{t_0}^t (\|\hta_1\|_1^2+\|\hta_2\|_1^2)~ds \le
K(t)H^8+K\int_{t_0}^t \|\bta_1(s)\|^2ds. $$
Apply Gronwall's lemma to establish $L^{\infty}(\bL^2)$-norm estimate of $\bta_1$.
We note that
$$ \|\bta_1\|_1 \le cH^{-1}\|\bta_1\| \le K(t)H^3. $$
For $\bta_2$, we again put $\bchi=e^{2\alpha t}\bta_2$ in (\ref{bta12}).
\begin{align}\label{l2.err04}
\nu \|\hta_2\|_1^2= +e^{2\alpha t}\Lambda_{h,2}(\bta_2)-\nu a(\hta_1,\hta_2).
\end{align}
Recall that
$$ \Lambda_{h,2}(\bta_2)= b(\bxi-\bta,\bbu,\bta_2)+b(\bu_h,\bxi-\bta_1,\bta_2) +b(\bxi-\bta,\bta_1,\bta_2). $$
And
\begin{align*}
b(\bxi-\bta_1,\bbu,\bta_2)+b(\bu_h,\bxi-\bta_1,\bta_2) \le (\|\bxi\|+\|\bta_1\|)
(\|\bbu\|_2+\|\bu_h\|_2)\|\bta_2\|_1 \\
b(\bxi-\bta_1,\bta_1,\bta_2) \le (\|\bxi\|_1+\|\bta_1\|_1)\|\bta_1\|_1\|\bta_2\|_1 \\
b(-\bta_2,\bbu,\bta_2)+b(-\bta_2,\bta_1,\bta_2) \le H(\|\bbu\|_2+\|\bta_1\|_2)\|\bta_2\|_1^2.
\end{align*}
Note that $\|\bta_1\| \le \|\by^H\|+\|\bby\|.$ And under the assumption
$$ \nu-H(\|\bbu\|_2+\|\by^H\|_2) \ge 0 $$
we easily obtain that
$$ \|\bta_2\|_1 \le K(t)H^3 $$
and hence
$$ \|\bta_2\| \le cH\|\bta_2\|_1 \le K(t)H^4. $$
Now, triangle inequality completes the proof.
\end{proof}

\subsection{NLGM II}

In this subsection, we deal with the error estimate for NLGM II.
As earlier, we split the error in two, that is, $\e=\bu_h-\bu^h=\bxi-\bta$. The equations and
hence the estimates regarding $\bxi$ remain same and are optimal in nature. The equation in
$\bta$ reads as follows:
\begin{align} \label{bta12a}
\left\{\begin{array}{rl}
(\bta_{1,t},\bphi) +& \nu a (\bta,\bphi)= b(\bu_h,\bu_h,\bphi)-b(\bu^h,\bu^h,\bphi),~~
\bphi\in\bJ_H \\
& \nu a (\bta,\bchi)= b(\bu_h,\bu_h,\bchi)-b(\bu^h,\bu^h,\bchi)+b(\bz^h,\bz^h,\bchi),~~
\bchi\in\bJ_h^H.
\end{array}\right.
\end{align}
\begin{lemma}\label{l2.err1}
Under the assumptions of Lemma \ref{l2.err}, we have
$$ \|(\bu_h-\bu^h)(t)\| \le K(t)H^3. $$
\end{lemma}

\begin{proof}
We choose $\bphi=e^{2\alpha t}\bta_1,~\bchi=e^{2\alpha t}\bta_2$ in (\ref{bta12a}).
\begin{align}\label{l2.err101}
\frac{1}{2}\frac{d}{dt}\|\hta_1\|^2+\nu \|\hta\|_1^2= e^{2\alpha t}\big\{\Lambda_h(\bta_1,\bta_2)
+b(\bz^h,\bz^h,\bta_2)\big\},
\end{align}
Since $\bz^h=\bz_h+\bta_2-\bxi_2,$ we have
$$ b(\bz^h,\bz^h,\bta_2)=b(\bz_h+\bta_2-\bxi_2,\bz_h-\bxi_2,\bta_2). $$
Now
\begin{align*}
b(\bxi_2,\bz_h-\bxi_2,\bta_2) &\le \|\bxi_2\|_1(\|\bz_h\|_1+\|\bxi_2\|_1)\|\bta_2\|_1 \\
b(\bta_2,\bz_h-\bxi_2,\bta_2) &=b(\bta_2,\bbz,\bta_2) \le cH\|\bbz\|_2\|\bta_2\|_1^2 \\
b(\bz_h,\bz_h-\bxi_2,\bta_2) & \le  \|\bz_h\|^{1/2}\|\bz_h\|_1^{1/2}(\|\bz_h\|_1
+\|\bxi_2\|_1)\|\bta_2\|^{1/2}\|\bta_2\|_1^{1/2} \\
&+\|\bz_h\|^{1/2}\|\bz_h\|_1^{1/2}\|\bta_2\|_1(\|\bz_h\|^{1/2}\|\bz_h\|_1^{1/2}
+\|\bxi_2\|_1) \\
& \le cH\|\bz_h\|_1(\|\bz_h\|_1+\|\bxi_2\|_1)\|\bta_2\|_1+cH^{1/2}\|\bz_h\|_1\|\bta_2\|_1.
cH^{1/2}(\|\bz_h\|_1+\|\bxi_2\|_1) \\
& \le cH\|\bz_h\|_1^2\|\bta_2\|_1+cH\|\bz_h\|_1\|\bxi_2\|_1\|\bta_2\|_1.
\end{align*}
Incorporate these in (\ref{l2.err101}). Integrate and as earlier, for small $H$, we obtain
$$ \|\bta_1\|^2+e^{-2\alpha t}\int_{t_0}^t (\|\hta_1\|_1^2+\|\hta_2\|_1^2)~ds \le
K(t)H^8+K(t)H^2\|\bz_h\|_1^4+K\int_{t_0}^t \|\bta_1(s)\|^2ds, $$
which results in
$$ \|\bta_1\|^2+e^{-2\alpha t}\int_{t_0}^t (\|\hta_1\|_1^2+\|\hta_2\|_1^2)~ds \le
K(t)H^6. $$
That is
$$ \|\bta_1\| \le K(t)H^3,~~\|\bta_1\|_1 \le K(t)H^2. $$
As in the previous section, using only the second equation of (\ref{bta12a}) we can easily conclude that
$$ \|\bta_2\| \le K(t)H^3,~~\|\bta_2\|_1 \le K(t)H^2. $$
\end{proof}
\noindent
\begin{remark}\label{rmk}
The analysis reveals that the decrease in the order of convergence is due to the presence of $b(\bz_h,\bz_h,\bchi)$ in the error equation. So, whereas in NLG I, we keep the nonlinearity
in both the equations, in NLG II, the second equation is made linear in $\bz^h$ by dropping
the term $b(\bz_h,\bz_h,\bchi)$ and which in turn appears in the error equation and is responsible for bringing down the rate of convergence in the above analysis.
\end{remark}

\subsection{Improved Error Estimate}

In this section, we try to improve the rate of convergence, using the technique of
Marion \& Xu \cite{MX95}. But this is not straightforward, as the estimate of the function $f(\bu)$ in their semi-linear problem does not hold for our $f(\bu)$ and we have to be careful in order to obtain similar results.

\noindent
First, we note that the second equation of (\ref{dwfJH1}) can be written as
\begin{equation}\label{relbyz} 
\bz^h=\Phi(\by^H),
\end{equation}
where $\Phi:\bJ_H\to \bJ_h^H$. Using this, we can write the equation in $\Phi(\by_H)$,
for $\bchi\in\bJ_h^H$
\begin{align}\label{eqn.phi1}
\nu a (\by_H+\Phi(\by_H),\bchi)+b(\by_H+\Phi(\by_H),\by_H,\bchi)+b(\by_H,\Phi(\by_H),\bchi)  =(\f, \bchi)
\end{align}
\begin{lemma}\label{err.phi} 
Under the assumptions of Lemma \ref{est.uh} and that $H$ is small enough to satisfy
$$ \frac{\nu}{2}-cH\|\bu_h\|_1 \ge 0, $$
we have
\begin{equation}\label{ep1}
\|\bz_h-\Phi(\by_H)\|+H\|\bz_h-\Phi(\by_H)\|_1 \le K(t)H^4.
\end{equation}
\end{lemma}

\begin{proof}
With the notation $\Phi_{\e}:= \bz_h-\Phi(\by_H)\in \bJ_h^H$, we have, by deducting
(\ref{eqn.phi1}) from the second equation of (\ref{dwfjH})
\begin{align}\label{eqn.pe}
\nu a(\Phi_{\e},\bchi)= -(\bz_{h,t},\bchi)-b(\bu_h,\bu_h,\bchi)+b(\by_H,\Phi(\by_H),\bchi) +b(\by_H+\Phi(\by_H),\by_H,\bchi).
\end{align}
Put $\bchi=\Phi_{\e}$ in (\ref{eqn.pe}) to obtain
\begin{align}\label{ep01}
\nu \|\Phi_{\e}\|_1^2 =-(\bz_{h,t},\Phi_{\e})-b(\Phi_{\e},\bu_h,\Phi_{\e})
-b(\Phi(\by_H),\Phi(\by_H),\Phi_{\e}).
\end{align}
Note that
\begin{align*}
-(\bz_{h,t},\Phi_{\e}) &\le \|\bz_{h,t}\|\|\Phi_{\e}\| \le K(t)H^3\|\Phi_{\e}\|_1 \\
-b(\Phi_{\e},\bu_h,\Phi_{\e}) &\le c\|\bu_h\|_1\|\Phi_{\e}\|\|\Phi_{\e}\|_1
\le cH\|\bu_h\|_1\|\Phi_{\e}\|_1^2 \\
-b(\Phi(\by_H),\Phi(\by_H),\Phi_{\e}) &= -b(\bz_h-\Phi_{\e},\bz_h,\Phi_{\e}) \\
&\le (\|\bz_h\|^{1/2}\|\bz_h\|_1^{1/2}+\|\Phi_{\e}\|^{1/2}\|\Phi_{\e}\|_1^{1/2})
\|\bz_h\|_1\|\Phi_{\e}\|^{1/2}\|\Phi_{\e}\|_1^{1/2} \\
&\le K(t)H^3\|\Phi_{\e}\|_1+cH\|\bz_h\|_1\|\Phi_{\e}\|_1^2.
\end{align*}
Put these estimates in (\ref{ep01}) to find
\begin{align}\label{ep02}
\frac{\nu}{2} \|\Phi_{\e}\|_1^2 \le K(t)H^6+cH\|\bu_h\|_1\|\Phi_{\e}\|_1^2.
\end{align}
We have used the fact that $\|\bz_h\| \le \|\bu_h\|+\|\by_H\| \le c\|\bu_h\|.$ And assuming
$H$ to be small enough to satisfy
$$ \frac{\nu}{2}-cH\|\bu_h\|_1 \ge 0 $$
we establish from (\ref{ep02})
$$ \|\Phi_{\e}\|_1 \le K(t)H^3. $$
And hence
$$ \|\Phi_{\e}\| \le cH\|\Phi_{\e}\|_1 \le K(t)H^4. $$
This completes the proof.
\end{proof}

\begin{lemma}\label{err.e2}
Under the assumptions of Lemma \ref{err.phi}, we have
\begin{eqnarray}
\|\e_2\|_1^2 \le K(t)H^6+K\|\e_1\|_1^2, \label{ee2.1} \\
\|\e_2\|^2 \le K(t)H^8+KH^2\|\e_1\|_1^2. \label{ee2.2}
\end{eqnarray}
\end{lemma}

\begin{proof}
Recall that $\e_2=\bz_h-\bz^h= (\bz_h-\Phi(\by_H))-(\bz^h-\Phi(\by_H))$. With the
notation $\Phi^{\e}=\bz^h-\Phi(\by_H)$, we have $\e_2=\Phi_{\e}-\Phi^{\e}.$ The
equation in $\Phi^{\e}$ can be obtained by deducting (\ref{eqn.phi1}) from the second
equation of (\ref{dwfJH1}).
\begin{align}\label{eqn.pe1}
\nu a(\Phi^{\e}-\e_1,\bchi)= -b(\bu^h,\by^H,\bchi)-b(\by^H,\bz^h,\bchi)
+b(\by_H,\Phi(\by_H),\bchi)+b(\by_H+\Phi(\by_H),\by_H,\bchi).
\end{align}
Put $\bchi=\Phi^{\e}$ to obtain
\begin{align}\label{ee201}
\nu \|\Phi^{\e}\|_1^2= \nu a(\e_1,\Phi^{\e})+b(\bu^h,\e_1,\Phi^{\e}) +b(\e_1-\Phi^{\e},\by_H,\Phi^{\e})+b(\e_1,\Phi(\by_H),\Phi^{\e})
\end{align}
Note that
\begin{align*}
b(\bu^h,\e_1,\Phi^{\e})&=b(\bu_h-\e_1-\Phi_{\e},\e_1,\Phi^{\e})+b(\Phi^{\e},\e_1,\Phi^{\e}) \\
b(\bu_h-\e_1-\Phi_{\e},\e_1,\Phi^{\e}) &\le \|\bu_h\|_1\|\e_1\|_1\|\Phi^{\e}\|_1 +\|\e_1\|_1^2\|\Phi^{\e}\|_1+\|\Phi_{\e}\|_1\|\e_1\|_1\|\Phi^{\e}\|_1 \\
b(\Phi^{\e},\e_1,\Phi^{\e}) &= \frac{1}{2}(\Phi^{\e}\cdot\nabla\e_1,\Phi^{\e})-\frac{1}{2}
(\Phi^{\e},\cdot\nabla\Phi^{\e},\e_1) \\
&\le \|\Phi^{\e}\|\|\Phi^{\e}\|_1\|\e_1\|_1+\|\e_1\|_{\infty}\|\Phi^{\e}\|_1\|\Phi^{\e}\|
\le cH(1+L_H)\|\e_1\|_1\|\Phi^{\e}\|_1^2 \\
b(\e_1-\Phi^{\e},\by_H,\Phi^{\e}) &=b(\e_1-\Phi^{\e},\bu_h-\bz_h,\Phi^{\e}) \\
&\le \|\e_1\|_1\|\bu_h\|_1\|\Phi^{\e}\|_1+\|\e_1\|_1\|\bz_h\|_1\|\Phi^{\e}\|_1
+cH(1+L_H)\|\by_H\|_1\|\Phi^{\e}\|_1^2 \\
b(\e_1,\Phi(\by_H),\Phi^{\e}) &\le \|\e_1\|_1(\|\Phi_{\e}\|_1+\|\bz_h\|_1)\|\Phi^{\e}\|_1.
\end{align*}
Incorporate these estimates in (\ref{ee201}). Whenever it suits us, we bound $\|\e_1\|_1$ by
$\|\by_H\|_1+\|\by^H\|_1 \le K$. And therefore, we have, after kickback argument
\begin{align}\label{ee202}
\frac{\nu}{2}\|\Phi^{\e}\|_1^2 \le K\|\e_1\|_1^2+K(t)H^6+cH(1+L_H)(\|\by_H\|_1+\|\by^H\|_1)
\|\Phi^{\e}\|_1^2.
\end{align}
Assuming $H$ small enough to satisfy
$$ \frac{\nu}{2}-cH(1+L_H)(\|\by_H\|_1+\|\by^H\|_1) > 0, $$
we obtain
$$ \|\Phi^{\e}\|_1^2 \le K(t)H^6+K\|\e_1\|_1^2. $$
And so
$$ \|\Phi^{\e}\|^2 \le K(t)H^8+KH^2\|\e_1\|_1^2. $$
Using triangle inequality, we complete the proof.
\end{proof}
\begin{remark}\label{eta.subopt}
Recall that $\e_i=\bxi_i-\bta_i,~i=1,2$ and since the linearized error $\bxi$ (that is, $\bxi_1,\bxi_2$) is optimal in nature, we have from (\ref{ee2.1})-(\ref{ee2.2})
\begin{eqnarray}
\|\bta_2\|_1^2 \le K(t)H^6+c\|\bta_1\|_1^2 \label{eta.H1} \\
\|\bta_2\|^2 \le K(t)H^8+cH^2\|\bta_1\|_1^2. \label{eta.L2}
\end{eqnarray}
\end{remark}
\noindent
Following Marion \& Xu \cite{MX95}, we introduce the operator $R_h^H:\bJ_h\to \bJ_h^H$
satisfying
\begin{equation}\label{prop.R}
a\big(\bv-R_h^H\bv,\bchi\big)=0,~\forall\chi\in\bJ_h^H.
\end{equation}
With the notations
$$ ||\bv||_R=\|(I-R_h^H)\bv\|_1,~~~ (\bv,\bw)_R=a\big((I-R_h^H)\bv,(I-R_h^H)\bw\big), $$
we have, from Lemma 4.1 of \cite{MX95},
\begin{equation}\label{norm.R}
c_1\|\bv\|_1 \le \|\bv\|_R \le c_2\|\bv\|_1,
\end{equation}
where $c_1,c_2$ are positive constants independent of $h,H$.
And similar to Lemmas 4.5 and 4.6 of \cite{MX95}, we find for $\bphi\in\bJ_H$
\begin{eqnarray}
(\by_t^H,\bphi)+\nu (\by^H,\bphi)_R &=& (\f,(I-R_h^H)\bphi)-b(\bu^h,\bu^h,(I-R_h^H)\bphi) -b(\bz^h,\bz^h,R_h^H\bphi) \label{eqn.R} \\
(\by_{H,t},\bphi)+\nu (\by_H,\bphi)_R &=& (\f,(I-R_h^H)\bphi)-b(\bu_h,\bu_h,(I-R_h^H)\bphi) +(\bu_{h,t},R_h^H\bphi).
\end{eqnarray}
Now, for $\bphi\in\bJ_H$, we write the equation in $\e_1=\by_H-\by^H$ as
\begin{align}\label{eqn.e1.R}
(\e_{1,t},\bphi)+\nu (\e_1,\bphi)_R= +(\bu_{h,t},R_h^H\bphi)+b(\bz^h,\bz^h,R_h^H\bphi)
-b(\e_1+\e_2,\bu_h,(I-R_h^H)\bphi) \nonumber \\
-b(\bu^h,\e_1+\e_2,(I-R_h^H)\bphi).
\end{align}
\begin{lemma}\label{err.e1.h1}
Under the assumptions Lemma \ref{err.phi}, we have
$$ \|\e_1\|_1^2+\int_{t_0}^t \|\e_{1,t}\|^2ds \le K(t)H^6. $$
\end{lemma}

\begin{proof}
Put $\bphi=\e_{1,t}$ in (\ref{eqn.e1.R}) and observe that
\begin{align*}
(\bu_{h,t},R_h^H\e_{1,t}) &=\frac{d}{dt}(\bu_{h,t},R_h^H\e_1)-(\bu_{h,tt},R_h^H\e_1) \\
&=\frac{d}{dt}(\bu_{h,t},R_h^H\e_1)-((I-P_H)\bu_{tt},R_h^H\e_1)-((\bu_h-\bu)_{tt},
R_h^H\e_1) \\
& \le \frac{d}{dt}(\bu_{h,t},R_h^H\e_1)+K(t)H^3\|\e_1\|_1, \\
-b(\e_1+\e_2,\bu_h,& (I-R_h^H)\e_{1,t})-b(\bu^h,\e_1+\e_2,(I-R_h^H)\e_{1,t}) \\
& \le c(\|\e_1\|_1+\|\e_2\|_1)(\|\bu_h\|_2+\|\bu^h\|_2)\|\e_{1,t}\| \\
b(\bz^h,\bz^h,R_h^H\e_{1,t}) &=b(\bz_h-\e_2,\bz_h-\e_2,R_h^H\e_{1,t})
= \frac{d}{dt} b(\bz_h,\bz_h,R_h^H\e_1)-b(\bz_{h,t},\bz_h,R_h^H\e_1) \\
& -b(\bz_h,\bz_{h,t},R_h^H\e_1)+b(\bz^h,-\e_2,R_h^H\e_{1,t})+b(\e_2,\bz_h,R_h^H\e_{1,t}) \\
-b(\bz_{h,t},\bz_h,R_h^H\e_1)&-b(\bz_h,\bz_{h,t},R_h^H\e_1) \le cH\|\bz_h\|_1 \|\bz_{h,t}\|_1\|\e_1\|_1 \\
b(\e_2,\bz_h,R_h^H\e_{1,t})& \le c\|\e_2\|^{1/2}\|\e_2\|_1^{1/2}(\|\bz_h\|_1
\|R_h^H\e_{1,t}\|^{1/2}\|R_h^H\e_{1,t}\|_1^{1/2}+\|\bz_h\|^{1/2}\|\bz_h\|_1^{1/2}
\|R_h^H\e_{1,t}\|_1) \\
&\le cH\|\e_2\|_1\|\bz_h\|_1\|\e_{1,t}\|_1 \le c\|\e_2\|_1\|\bz_h\|_1\|\e_{1,t}\|.
\end{align*}
Here, we have used that $\|(I-R_h^H)\e_{1,t}\| \le \|\e_{1,t}\|+cH\|\e_{1,t}\|_1
\le c\|\e_{1,t}\|$. And now we find
\begin{align}\label{e1h101}
\|\e_{1,t}\|^2+\frac{\nu}{2}\frac{d}{dt}\|\e_1\|_R^2 \le K(t)H^3\|\e_1\|_1+
\frac{d}{dt} \Big\{(\bu_{h,t},R_h^H\e_1)+b(\bz_h,\bz_h,R_h^H\e_1)\Big\} \nonumber \\
+K(\|\e_1\|_1+\|\e_2\|_1) \|\e_{1,t}\|+cH\|\bz_h\|_1 \|\bz_{h,t}\|_1\|\e_1\|_1
+c\|\e_2\|_1(\|\bz_h\|_1+\|\bz^h\|_1)\|\e_{1,t}\|.
\end{align}
Integrate (\ref{e1h101}), use (\ref{norm.R}) and the fact that $\e_1(t_0)=0$ to find
\begin{align*}
\|\e_1\|_1^2+\int_{t_0}^t \|\e_{1,t}\|^2ds \le K(t)H^6+c\int_{t_0}^t (\|\e_1\|_1^2
+\|\e_2\|_1^2)~ds+(\bu_{h,t},R_h^H\e_1)+b(\bz_h,\bz_h,R_h^H\e_1).
\end{align*}
As earlier, we estimate the last three terms to obtain
\begin{align}\label{e1h102}
\|\e_1\|_1^2+\int_{t_0}^t \|\e_{1,t}\|^2ds \le K(t)H^6+c\int_{t_0}^t (\|\e_1\|_1^2
+\|\e_2\|_1^2)~ds.
\end{align}
We note from (\ref{eta.H1}) and triangle inequality that
$$ \|\e_2\|_1^2 \le K(t)H^6+\|\bta_1\|_1^2 \le K(t)H^6+\|\e_1\|_1^2. $$
Therefore
\begin{align*} 
\|\e_1\|_1^2+\int_{t_0}^t \|\e_{1,t}\|^2ds \le K(t)H^6+c\int_{t_0}^t \|\e_1\|_1^2ds.
\end{align*}
Use Gronwall's lemma to establish
$$ \|\e_1\|_1^2+\int_{t_0}^t \|\e_{1,t}\|^2ds \le K(t)H^6. $$
\end{proof}
\begin{remark}
This tells us that
$$ \|\bta_1\|_1 \le K(t)H^3, $$
and as a result, from Remark \ref{eta.subopt}, we have
$$ \|\bta_2\|+H\|\bta_2\|_1 \le K(t)H^4. $$
Another application of triangle inequality results in
$$ \|\e_2\|+H\|\e_2\|_1 \le K(t)H^4. $$
\end{remark}
\noindent
For the final estimate, we note down the equations in terms of $\e_i,~i=1,2$.
\begin{align}\label{eqn.e12}
\left\{\begin{array}{rl}
(\e_{1,t},\bphi)+\nu a (\e,\bphi)=& -b(\bu_h,\bu_h,\bphi)+b(\bu^h,\bu^h,\bphi) \\
\nu a (\e,\bchi)=& -(\bz_{ht},\bchi)-b(\bu_h,\bu_h,\bchi)
+b(\bu^h,\bu^h,\bchi)-b(\bz^h,\bz^h,\bchi),
\end{array}\right.
\end{align}

\begin{lemma}\label{err.e1}
Under the assumptions of Lemma \ref{err.phi}, we have
$$ \|(\bu_h-\bu^h)(t)\| \le K(t)H^4,~~t>t_0. $$
\end{lemma}

\begin{proof}
With the notation $\td_H=P_H(-\Delta_h)$, we choose $\phi=\td_H^{-1}\e_{1,t}$ in the
first equation of (\ref{eqn.e12}) to find
\begin{align}\label{ee101}
\|\e_{1,t}\|_{-1}^2+\frac{\nu}{2}\frac{d}{dt}\|\e_1\|^2= b(\e_1+\e_2,\bu_h,\td_H^{-1}\e_{1,t}) +b(\bu^h,\e_1+\e_2,\td_H^{-1}\e_{1,t}).
\end{align}
As earlier, we have
\begin{align*}
b(\e_1+\e_2,\bu_h,\td_H^{-1}\e_{1,t})+b(\bu^h,\e_1+\e_2,\td_H^{-1}\e_{1,t})
\le c(\|\e_1\|+\|\e_2\|)(\|\bu_h\|_2+\|\bu^h\|_2)\|\e_{1,t}\|_{-1}.
\end{align*}
Integrate (\ref{ee101}) and use the above estimate to find
\begin{align*}
\|\e_1\|^2+\int_{t_0}^t \|\e_{1,t}\|_{-1}^2 \le c\int_{t_0}^t (\|\e_1\|^2+\|\e_2\|^2)~ds
\le K(t)H^8+c\int_{t_0}^t \|\e_1\|^2ds.
\end{align*}
Apply Gronwall's lemma to conclude
$$ \|\e_1\|^2+\int_{t_0}^t \|\e_{1,t}\|_{-1}^2 \le K(t)H^8. $$
\end{proof}

\begin{remark}\label{rmk1}
It is clear from our above analysis is that the linearized error between NLG approximation and
Galerkin approximation is of order $H^4$ in $L^2$-norm. However, non-linearized part of the
error may not always be of same order. For example, if the equation in $\bz^h$ contains only
$b(\by^H,\by^H,\bchi)$, then the non-linearized part of the error (i.e. the equation in $\bta$)
will contain additional terms like $b(\by^H,\bz^h,\bchi)$ and $b(\bz^h,\by^H,\bchi)$
apart from the non-linear terms of the second equation of (\ref{bta12a}). And with one of these
terms, we believe, we can only manage $H^3$ order of convergence in $L^2$-norm.
\end{remark}

\section{Summary}

In this work, our main focus is in obtaining optimal $L^2$-error estimate , that is,
$\mathcal{O}(H^4)$ for nonlinear Galerkin finite element approximations. For that purpose, we
have discussed two NLGMs. In the first one, small scales are assumed stationary. In the second
one, we have an additional assumption that interactions between small scales are negligible.
And in both these cases, we have managed to show optimal $L^2$-error estimate. But any further
simplification of the small scales equations will lead to sub-optimal error estimate, say
$\mathcal{O}(H^3)$, as has been observed in the remark \ref{rmk1}.
\vspace*{0.2cm} \\

\noindent
{\bf Acknowledge}: The author would like to thank CAPES/INCTMat, Brazil for financial grant.

\end{document}